\documentclass[12pt]{amsart}

\usepackage{tikz}
\usetikzlibrary{snakes}
\usepackage[colorlinks=true, linkcolor=blue, anchorcolor=blue, citecolor=blue, filecolor=blue, menucolor= blue, urlcolor=blue]{hyperref}
\usepackage{amsmath,amssymb,amsthm,amscd}
\usepackage{verbatim}
\usepackage{comment}
\usepackage{multirow}
\usepackage{mathtools}
\usepackage{enumitem}
\usepackage{pgf,tikz}
\usetikzlibrary{arrows}
\usepackage[normalem]{ulem}
\usepackage{graphicx}
\usepackage{amssymb}
\usepackage{wasysym}

\usepackage{arydshln}
\usepackage{faktor}
\usepackage{xfrac}  

\usepackage[margin=1in]{geometry} 
\usepackage{pdfpages}

\numberwithin{equation}{section}

\newtheorem{theorem}{Theorem}[section]
\newtheorem{proposition}[theorem]{Proposition}
\newtheorem{lemma}[theorem]{Lemma}
\newtheorem{corollary}[theorem]{Corollary}

\newtheorem{question}[theorem]{Question}
\newtheorem{conjecture}[theorem]{Conjecture}
\newtheorem*{theorem*}{Theorem}

\theoremstyle{definition}

\newtheorem{example}[theorem]{Example}
\newtheorem{remark}[theorem]{Remark}

\newcommand{\CC}{ \ensuremath{\mathbb{C}}}

\newcommand{\PP}{ \ensuremath{\mathbb{P}}}
\newcommand{\KK}{ \ensuremath{\mathbb{K}}}

\definecolor{MyDarkGreen}{cmyk}{0.7,0,1,0}

\def\cocoa{{\hbox{C\kern-.13em o\kern-.07em C\kern-.13em o\kern-.15em A}}}

\begin{document}

\title[Values of resurgences]{Extreme values of the resurgence for homogeneous ideals in polynomial rings}

\author{B.\ Harbourne}
\address{Department of Mathematics\\
University of Nebraska\\
Lincoln, NE 68588-0130 USA}
\email{bharbourne1@unl.edu}

\author{J.\ Kettinger}
\address{Department of Mathematics\\
University of Nebraska\\
Lincoln, NE 68588-0130 USA}
\email{jkettinger@huskers.unl.edu}

\author{F.\ Zimmitti}
\address{Department of Mathematics\\
University of Nebraska\\
Lincoln, NE 68588-0130 USA}
\email{frank.zimmitti@huskers.unl.edu}

\begin{abstract}
We show that two ostensibly different versions of the asymptotic resurgence
introduced by Guardo, Harbourne and Van Tuyl in 2013 are the same.
We also show that the resurgence and asymptotic resurgence
attain their maximal values simultaneously, if at all,
which we apply to a conjecture of Grifo.
For radical ideals of points, we show that 
the resurgence and asymptotic resurgence
attain their minimal values simultaneously.
In addition, we introduce an integral closure version of the resurgence and relate it to the
other versions of the resurgence. In closing we provide various examples
and raise some related questions, and we finish with some remarks about
computing the resurgence.
\end{abstract}

\date{edited: May 11, 2020; compiled \today}

\thanks{
{\bf Acknowledgements}: Harbourne was 
partially supported by Simons Foundation grant \#524858.
We also wish to thank M. DiPasquale, E. Grifo, C. Huneke 
and A. Seceleanu, and
M. Dumnicki, H. Tutaj-Gasi\'nska, T. Szemberg and J. Szpond, for helpful comments.
}



\maketitle



\section{Introduction}
This paper is motivated by wanting to better understand concepts, conjectures and methods introduced in
\cite{BH1}, \cite{GHVT}, \cite{G}, \cite{DFMS} and \cite{DD}
involving various approaches to the containment problem.
It makes particular use of the groundbreaking methods of
\cite{DFMS} and \cite{DD}.

Let $\KK$ be a field and let $N\geq1$. Then $R=\KK[\PP^N]$ denotes the polynomial ring $\KK[\PP^N]=\KK[x_0,\ldots,x_N]$.
Now let $(0)\neq I\subsetneq \KK[\PP^N]$ be a homogeneous ideal;
thus $I=\oplus_{t\geq0} I_t$, where $I_t$ is the $\KK$ vector space span of all homogeneous polynomials of degree $t$ in $I$.
The symbolic power $I^{(m)}$ is defined as
$$I^{(m)}=R\cap (\cap_{P\in{\rm Ass}(R/I)}I^mR_P)$$
where the intersections take place in $\KK(\PP^N)$.

While $I^r\subseteq I^{(m)}$ holds if and only if $m\leq r$, the containment problem of
determining for which $m$ and $r$
the containment $I^{(m)}\subseteq I^r$ holds is much more subtle. 
If $h_I$ is the minimum of $N$ and the bigheight of $I$ (i.e., the maximum of the heights of associated 
primes of $I$), it is known that
\begin{equation}\label{ELSHH}
I^{(rh_I)}\subseteq I^r
\end{equation}
\cite{ELS, HoHu}, so given $r$, the issue is for which $m$ with
$m<rh_I$ do we have $I^{(m)}\subseteq I^r$.
The {\it resurgence} $\rho(I)$, introduced in \cite{BH1},
gives some notion of how small the ratio $m/r$ can be and still be sure to have
$I^{(m)}\subseteq I^r$; specifically,
$$\rho(I)=\sup\Big\{\frac{m}{r} : I^{(m)}\not\subseteq I^r\Big\}.$$

A case of particular interest is that of ideals of {\em fat points}.
Given distinct points $p_1,\ldots,p_s\in\PP^N$ and nonnegative integers $m_i$
(not all 0), let $Z=m_1p_1+\cdots+m_sp_s$ 
denote the scheme (called a fat point scheme) defined by the ideal
$$I(Z)=\cap_{i=1}^s(I(p_i)^{m_i})\subseteq \KK[\PP^N],$$
where $I(p_i)$ is the ideal generated by all homogeneous polynomials
vanishing at $p_i$. Note that $I(Z)$ is always nontrivial (i.e., not $(0)$ nor $(1)$).
Symbolic powers of $I(Z)$ take the form
$I(Z)^{(m)}=I(mZ)=\cap_{i=1}^s(I(p_i)^{mm_i})$.
We say $Z$ is {\it reduced} if $m_i$ is either 0 or 1 for each $i$
(i.e., if $I(Z)$ is a radical ideal).

Subsequent to \cite{BH1}, two asymptotic notions of the resurgence were introduced by \cite{GHVT}.
The first is 
$$\rho'(I)=\limsup_t\rho(I,t)=\lim_{t\to\infty}\rho(I,t),$$
where $\rho(I,t)=\sup\Big\{\frac{m}{r} : I^{(m)}\not\subseteq I^r, m\geq t, r\geq t\Big\}$.
The second is 
$$\widehat{\rho}(I)=\sup\Big\{\frac{m}{r} : I^{(mt)}\not\subseteq I^{rt}, t\gg0\Big\}.$$
A useful new perspective on $\widehat{\rho}(I)$ is given by \cite[Corollary 4.14]{DFMS}, which shows that
$$\widehat{\rho}(I)=\sup\Big\{\frac{m}{r} : I^{(m)}\not\subseteq \overline{I^{r}}\Big\},$$
where $\overline{I^{r}}$ is the integral closure of ${I^{r}}$ (defined below).
Since we always have $I^r\subseteq \overline{I^r}$, this new perspective makes it
clear that we always have $\widehat{\rho}(I)\leq\rho(I)$, and that we have $\widehat{\rho}(I)=\rho(I)$
if $I^r= \overline{I^r}$ for all $r\geq1$.

Our major results are to show that $\widehat{\rho}(I) = \rho'(I)$ (Theorem \ref{MT1}),
that $\widehat{\rho}(I)=h_I$ if and only if $\rho(I)=h_I$ (Theorem \ref{MT2}), and that
$\widehat{\rho}(I(Z))=1$ if and only if $\rho(I(Z))=1$ when $Z\subset\PP^N$ is a reduced scheme of points (Theorem \ref{MT3})
and for every fat point subscheme $Z\subset\PP^2$ (Corollary \ref{CorIntRho1}).
We also introduce a new version of the resurgence, $\rho_{int}$, based on integral closure,
and relate it to the original resurgence. We then discuss
the relevance of our results to a conjecture of Grifo.
Finally we provide some examples and raise some questions,
and include a discussion of the computability of the resurgence.

\subsection{Background}\label{bkgrnd}
Let $I\subseteq \KK[\PP^N]$ be a nontrivial homogeneous ideal. 
Given the comments and definitions above we have (see \cite{GHVT}) 
\begin{equation}\label{ELSHHbounds}
1\leq\widehat{\rho}(I)\leq \rho'(I)\leq \rho(I)\leq h_I.
\end{equation}
By \cite{BH1} and \cite{GHVT} we also have
\begin{equation}\label{BHbounds}
\frac{\alpha(I)}{\widehat{\alpha}(I)}\leq\widehat{\rho}(I)\leq \rho'(I)\leq \rho(I),
\end{equation}
where $\alpha(I)$ is the least degree of a nonzero element of $I$ and
$\widehat{\alpha}(I)$ is the {\it Waldschmidt constant}, defined as
$$\widehat{\alpha}(I)=\lim_{m\to\infty}\frac{\alpha(I^{(m)})}{m}.$$
For a nontrivial fat point subscheme $Z\subseteq\PP^N$, by \cite{BH1} we have in addition
\begin{equation}\label{BHbounds2}
\rho(I(Z))\leq \frac{{\rm reg}(I(Z))}{\widehat{\alpha}(I(Z))},
\end{equation}
where ${\rm reg}(I(Z))$ is the Castelnuovo-Mumford regularity of $I(Z)$.

A version of resurgence can be defined with integral closure replacing symbolic powers.
We pause to briefly discuss the concept of integral closure.
Given an ideal $I\subseteq R=\KK[\PP^N]$, we recall (see \cite{HuSw}) that the {\it integral closure} $\overline{I}$ of $I$
consists of all elements $c\in R$ satisfying for some $n\geq1$ a polynomial equation
$$c^n+a_1c^{n-1}+\cdots+a_n=0$$
where $a_j\in I^j$. We say $I$ is {\it integrally closed} if $I=\overline{I}$. 
We note that $\overline{I}$ is monomial (resp. homogeneous) if $I$ is \cite[Proposition 1.4.2]{HuSw} 
(resp. \cite[Corollary 5.2.3]{HuSw}).
If $I^r$ is integrally closed for all $r\geq1$, we say $I$ is {\it normal}.

For example, the ideal $I(p_i)$ of a point $p_i\in\PP^N$ is normal. Likewise,
$M=(x_0,\ldots,x_N)$ is normal. This is because $M$ is a monomial prime ideal
and $I(p_i)$ is also, up to choice of coordinates, but monomial primes are normal.
(Apply the usual criterion for integral closure for monomial ideals, that 
the integral closure of a monomial ideal $I$ is the monomial ideal
associated to the convex hull of the Newton polygon of $I$ \cite{HuSw}.)
As further examples of integrally closed ideals, 
we note $I(Z)$ is integrally closed for all $Z$, as is $M^t\cap I(Z)$
for every $t$; this is because intersections of integrally closed ideals are integrally closed.
Now assume $\alpha(I(Z))={\rm reg}(I(Z))$; then
$(I(Z)^r)_t=(I(rZ))_t$ for $t\geq \alpha(I(Z)^r)$ (apply \cite[Lemma 2.3.3(c)]{BH1}
using the fact that $\alpha(I(Z)^r)=r\alpha(I(Z))$).
Thus we have $I(Z)^r=M^{r\alpha(I)}\cap I(rZ)$ and hence
$I(Z)$ is normal, if $\alpha(I(Z))={\rm reg}(I(Z))$. Thus, when $\alpha(I(Z))={\rm reg}(I(Z))$,
we have $\widehat{\rho}(I(Z)) =\rho(I(Z))$ by the normality, but in fact
\eqref{BHbounds} and \eqref{BHbounds2} give us more, namely 
$$\frac{\alpha(I(Z))}{\widehat{\alpha}(I(Z))}=\widehat{\rho}(I(Z)) =\rho(I(Z)).$$

Now we define the integral closure resurgence.
Given any nontrivial homogeneous ideal $I\subset\KK[\PP^N]$ but
replacing symbolic power by integral closure in the definition of resurgence
gives us the {\it integral closure resurgence},
$$\rho_{int}(I)=\sup\Big\{\frac{m}{r} : \overline{I^m}\not\subseteq I^r\Big\}.$$
If symbolic powers of $I$ are integrally closed
(as is the case when $I$ is the ideal of a fat point subscheme), we have
$\rho_{int}(I)\leq\rho(I)$, and if moreover 
$I^{(m)}=\overline{I^m}$ for all $m\geq1$, it follows that $\rho_{int}(I)=\rho(I)$
(see Theorem \ref{BSkoda}).
A lower bound such as $\rho_{int}(I)\leq\rho(I)$ is of interest since for any $c>1$
(as the proof of Theorem \ref{BSkoda} shows) it is in principle a finite
calculation (although not necessarily an easy one)
to verify whether or not $\rho_{int}(I)\geq c$, and if so
to compute $\rho_{int}(I)$ exactly.
In a further analogy of $\rho_{int}$ with $\rho$, by Lemma \ref{rhoIntLem}
we have $1\leq\rho_{int}(I)\leq N$, and, if $I=\overline{I}$, we have
$\overline{I^m}\subseteq I^r$ for all $m\geq Nr$.

\section{Main results}
We begin with the result that started this paper.

\begin{proposition}
Let $Z$ be a fat point subscheme of $\PP^N$ with an integer $c$ such that $I(cZ)^t=I(ctZ)$ for all $t\geq 1$.
Then $\widehat{\rho}(I(Z)) =\rho'(I(Z))$.
\end{proposition}

\begin{proof}
Let $b$ be a rational such that $\widehat{\rho}(I(Z))<c/b$. Pick any integer $d>0$ such that $db$ is an integer.
Since $\widehat{\rho}(I(Z))<c/b=cd/(db)$, we have by definition of $\widehat{\rho}(I(Z))$ that $I(cdtZ)\subseteq I(Z)^{dbt}$ for $t\gg0$,
Note that $I(cdZ)^t=I(cZ)^{dt}=I(cdtZ)$ for all $t\geq1$. 
Hence by \cite[Theorem 1.2(3)]{GHVT} we have
$\rho'(I(Z))\leq cdt/(dbt)=c/b$. Since this holds for all $b$ with $c/b>\widehat{\rho}(I(Z))$, we have
$\widehat{\rho}(I(Z))\geq\rho'(I(Z))$. But by \cite[Theorem 1.2(1)]{GHVT} we have
$\widehat{\rho}(I(Z))\leq\rho'(I(Z))$, hence $\widehat{\rho}(I(Z))=\rho'(I(Z))$.
\end{proof}

A much stronger result can be proved based on an argument in \cite{DD}.

\begin{theorem}\label{MT1}
Let $I$ be a nontrivial homogeneous ideal of $\KK[\PP^N]$.
Then $\widehat{\rho}(I) =\rho'(I)$.
\end{theorem}

\begin{proof}
If $\widehat{\rho}(I) =\rho(I)$, then \eqref{ELSHHbounds} gives $\widehat{\rho}(I) =\rho'(I)$,
so assume $\widehat{\rho}(I) < \rho(I)$.
Thus there is an $\epsilon>0$
such that $\widehat{\rho}(I) + \epsilon < \rho(I)$.
As in the proof of \cite[Proposition 2.6]{DD}
(which in turn is a consequence of \cite[Lemma 4.12]{DFMS}),
we have 
$$\frac{s}{r+N} < \widehat{\rho}(I)$$
whenever $I^{(s)}\not\subseteq I^r$.

This means there are only finitely many $s$ and $r$ for which 
$\widehat{\rho}(I)+\epsilon\leq s/r$ holds but
$I^{(s)}\subseteq I^r$ fails.
(This is because $\widehat{\rho}(I) + \epsilon \leq s/r$ and $\frac{s}{r+N} < \widehat{\rho}(I)$
implies $r(\widehat{\rho}(I)+\epsilon)\leq s< \widehat{\rho}(I)(r+N)$ and $r < N\widehat{\rho}(I)/\epsilon$.)

Hence for all $s$ and $r$ sufficiently
large which have $I^{(s)}$ not contained in $I^r$, we will have 
$s/r < \widehat{\rho}(I)+\epsilon$ and hence 
$\rho'(I) \leq \widehat{\rho}(I)+\epsilon$. This is true for every $\epsilon > 0$, 
and so we get $\rho'(I) \leq \widehat{\rho}(I)$.
Since we already have $\rho'(I) \geq \widehat{\rho}(I)$, we conclude 
that $\rho'(I) = \widehat{\rho}(I)$.
\end{proof}

An alternate statement of the next result is that $\widehat{\rho}(I)=h_I$ if and only if $\rho(I)=h_I$.
We have learned that Theorem \ref{MT2} was also obtained independently by DiPasquale and Drabkin, 
but not included in \cite{DD}.

\begin{theorem}\label{MT2}
Let $I$ be a nontrivial homogeneous ideal of $\KK[\PP^N]$.
Then $\widehat{\rho}(I)<h_I$ if and only if $\rho(I)<h_I$.
\end{theorem}

\begin{proof}
If $\rho(I) < h_I$, then we have $\widehat{\rho}(I) \leq \rho(I) < h_I$ by \eqref{ELSHHbounds}.
So assume $\widehat{\rho}(I) < h_I$. If $\widehat{\rho}(I)=\rho(I)$, then $\rho(I) < h_I$.
If $\widehat{\rho}(I) < \rho(I)$, then by \cite[Proposition 2.6]{DD}, $\rho(I)$ is the maximum
of finitely many ratios $s/r$ with $I^{(s)}\not\subseteq I^r$.
Thus $\rho(I) = h_I$ would imply that $s/r = h_I$, which contradicts
the result of \cite{ELS, HoHu} that $I^{(m)}\subseteq I^r$
whenever $m\geq rh_I$.
\end{proof}

The other extreme is also of interest. Let $Z\subset\PP^N$ be a fat point subscheme. 
The next result, Theorem \ref{MT3}, shows that the question of when $\rho(I(Z))=1$
is related to two concepts: to analytic spreads (see \cite{HuSw}) and to
symbolic defects (see \cite{GGSV}). 

We can define the {\it analytic spread} $\ell(I)$ of a homogeneous ideal $I\subseteq R=\KK[\PP^N]$ as being the minimum
number of elements of $I$ such that, after localizing at the irrelevant ideal, the ideal $J$ they generate
has $\overline{J}=\overline{I}$ \cite[Corollary 1.2.5, Proposition 8.3.7]{HuSw}. 
The analytic spread $\ell(I)$ of an ideal $I\subseteq R$ is at least the height of $I$, since
$J$ and $I$ have the same minimal primes, and the minimal number of generators of an ideal 
is at least the height of its minimal prime of minimal height.
By \cite[Proposition 5.1.6]{HuSw}, it is also at most the dimension of $R$ (which is $N+1$).
For $I(Z)\subset \KK[\PP^N]$ we thus have $N\leq\ell(I(Z))\leq N+1$.

We say an ideal $B\subseteq\KK[\PP^N]$ is a {\it complete intersection}
if $B$ is generated by a regular sequence; equivalently, $B$ is a complete intersection
if all associated primes of $B$ have height $c$, where $c$ is the minimal number of
generators of $B$. If for a fat point subscheme $Z\subset\PP^N$ its ideal $I(Z)$ is a 
complete intersection (i.e., $I(Z)$ has $N$ generators), 
then $\ell(I(Z))=N$, but $\ell(I(Z))=N$ can occur even when $I(Z)$ is not a complete intersection
(for example, $\ell(I(mZ))=N$ when $I(Z)$ is a complete intersection but $m>1$).

The {\it symbolic defect} ${\rm sdefect}(I(Z),m)$ is the minimum number of generators of the module
$I(mZ)/I(Z)^m$ \cite{GGSV}. Thus ${\rm sdefect}(I(Z),m)=0$ if and only if $I(mZ)=I(Z)^m$.
If $I(Z)$ is a complete intersection, then ${\rm sdefect}(I(Z),m)=0$ for all $m\geq1$
since we have $I(mZ)=I(Z)^m$ (indeed, if $I$ is generated by a regular sequence, then
$I^r=I^{(r)}$ for all $r\geq 0$ by \cite[Lemma 5, Appendix 6]{ZS}). 
When $N=2$ and $Z$ is reduced (i.e., $I(Z)$ is radical), \cite[Theorem 2.6]{GGSV}
gives a converse:
if ${\rm sdefect}(I(Z),m)=0$ for all $m\geq1$, then $I(Z)$ is a complete intersection
(see also \cite[Remark 2.5]{CFGLMNSSV} and \cite[Theorem 2.8]{HU}). 

Our next result gives a number of equivalent conditions for ${\rm sdefect}(I(Z),m)=0$ for all $m\geq1$ for any 
reduced fat point subscheme $Z\subset \PP^N$, thereby extending from $N=2$ to all $N$ 
the result that $I(mZ)=I(Z)^m$ for all $m\geq1$ if and only if $I(Z)$ is a complete intersection.
We thank Seceleanu and Huneke for the implications (d) $\Rightarrow$ (e), and, for $Z$ reduced, (e) $\Rightarrow$ (a).
We do not know if (e) $\Rightarrow$ (a) holds for nonreduced fat point schemes $Z$ (i.e., when $I(Z)$ is not radical),
but see Corollaries \ref{CorIntRho1} and \ref{CorIntRho2}, and also Examples \ref{Ex1} and \ref{Ex1b} and Question \ref{Q4b}.

\begin{theorem}\label{MT3}
Let $Z$ be a nontrivial fat point subscheme of $\PP^N$. 
Then each of the following criteria implies the next.
\begin{enumerate}
\item[(a)] $I(Z)^m=I(mZ)$ (i.e., ${\rm sdefect}(I(Z),m)=0$) for all $m\geq1$.
\item[(b)] $\rho(I(Z))=1$.
\item[(c)] $\widehat{\rho}(I(Z))=1$.
\vskip.3\baselineskip
\item[(d)] $\overline{I(Z)^m}=I(mZ)$ for all $m\geq1$.
\item[(e)] The analytic spread of $I(Z)$ is $N$.
\end{enumerate}
Moreover, if $Z$ is reduced or $N=1$, then (e) implies (a).
In fact, if $Z$ is reduced or $N=1$, then each of the conditions (a)-(e) 
is equivalent to $I(Z)$ being a complete intersection.
\end{theorem}

\begin{proof}
That (a) implies (b) is clear (since $I(Z)^m\subseteq I(Z)^r$ if and only if $m\geq r$), 
and (b) implies (c) since $1\leq\widehat{\rho}(I(Z))\leq\rho(I(Z))$.
Next (c) is equivalent to (d) by \cite[Corollary 4.16]{DFMS}.
Now we show (d) implies (e). By
\cite[Proposition 5.4.7]{HuSw}, we have that (d) implies $\ell(I(Z))\neq N+1$, hence
$\ell(I(Z)) = N$. 

If $N=1$, then $I(Z)$ is principal, hence (a) always holds and $I(Z)$ is a complete intersection.
Finally assume that $Z$ is reduced. Thus the primary components of $I(Z)$
are ideals of points of multiplicity 1, hence are complete intersections
(i.e., $I(Z)$ is locally a complete intersection). By \cite{CN}, $\ell(I(Z)) = N$
implies that the localization $I(Z)_M$ of $I(Z)$ at the irrelevant ideal $M\subset R=\KK[\PP^N]$ 
is a complete intersection, and hence $(I(Z)^m)_M=(I(Z)_M)^m$ is saturated for all $m\geq1$, so 
$I(Z)^m$ itself is saturated so (a) holds.
Moreover, since (e) implies that $I(Z)_M$ is a complete intersection when $Z$ is reduced, 
the number of generators of
$I(Z)_M$ is $N=\dim I(Z)_M/I(Z)_MM_M=\dim I(Z)/I(Z)M$, hence $I(Z)$ 
also has $N$ generators, so is itself is a complete intersection and hence (a) holds.
\end{proof}

We can now give a characterization of those $Z$ for which $I(mZ)=I(Z)^m$ for all $m\geq1$.
This characterization is not so interesting in itself, but it does raise the question
of whether the normality hypothesis can be dropped; see Question \ref{Q4b}.
Although $I(Z)$ is not always normal (see Example \ref{Ex2}),
we do not know any examples with $\widehat{\rho}(I(Z))=1$ where $I(Z)$ is not normal. 

\begin{corollary}\label{MT3Cor}
Let $Z\subset\PP^N$ be a fat point subscheme. 
Then $I(mZ)=I(Z)^m$ for all $m\geq1$ if and only if $\widehat{\rho}(I(Z))=1$ and $I(Z)$ is normal.
\end{corollary}

\begin{proof}
Having $I(mZ)=I(Z)^m$ for all $m\geq1$ implies $\widehat{\rho}(I(Z))=1$ by Theorem \ref{MT3},
and it implies $I(Z)$ is normal since $I(Z)^m\subseteq \overline{I(Z)^m}\subseteq I(mZ)$. Conversely,
$\widehat{\rho}(I(Z))=1$ implies $I(mZ)=\overline{I(Z)^m}$ for all $m\geq1$ 
by Theorem \ref{MT3} and normality implies $\overline{I(Z)^m}=I(Z)^m$ for all $m\geq1$,
hence $I(mZ)=I(Z)^m$ holds for all $m\geq1$.
\end{proof}

We now consider $\rho_{int}(I)$. Note for a nontrivial homogeneous ideal $I\subset\KK[\PP^N]$
that an ideal independent version of \eqref{ELSHH}
can be stated as $I^{(Nr)}\subseteq I^r$, and the corresponding bounds \eqref{ELSHHbounds} 
on $\rho(I)$ are $1\leq \rho(I)\leq N$. In analogy with this, we have the following lemma.
(See Example \ref{Ex4} for an example showing that the assumption $I=\overline{I}$ is needed both
for the second part of Lemma \ref{rhoIntLem} and for Theorem \ref{BSkoda}(b).)

\begin{lemma}\label{rhoIntLem}
Let $I\subset\KK[\PP^N]$ be a nontrivial homogeneous ideal.
Then $1\leq \rho_{int}(I)\leq N$. Moreover,
if $N>1$ and $t>1$, or if $I=\overline{I}$  and $t\geq1$, then $\overline{I^{Nt}}\subseteq I^t$.
\end{lemma}

\begin{proof}
Since $I^t\subseteq \overline{I^{t}}$ but $I^t\not\subseteq I^{t+1}$,
we see that $\overline{I^t}\not\subseteq I^{t+1}$, so $1\leq\rho_{int}(I)$.

The Brian{\c c}on-Skoda Theorem \cite{HuSw} asserts
$\overline{I^{t+N}}\subseteq I^t$ for each $t\geq1$.
For $t,N\geq2$ we have $tN\geq t+N$, hence
$I^{Nt}\subseteq I^{t+N}$ and thus
$\overline{I^{Nt}}\subseteq\overline{I^{t+N}}\subseteq I^t$.
This also tells us that if $\overline{I^m}\not\subseteq I^t$, we must have either
$m<Nt$ (and so $m/t<N$) or $t$ or $N$ must be equal to 1.
By Brian{\c c}on-Skoda, $\overline{I^m}\not\subseteq I^t$ implies $m<t+N$, so if $t=1$, then $m/t = m\leq N$,
while if $N=1$, then we have $m<t+1$, so $m/t\leq t/t=1=N$.
Thus in all cases we have $m/t\leq N$, hence $\rho_{int}(I)\leq N$.

We already saw that $\overline{I^{Nt}}\subseteq I^t$ if $N,t>1$. Assume $N>1$ but $t=1$.
Then $I^N\subseteq I$ so if $I=\overline{I}$ we have $\overline{I^{N}}\subseteq I$.
Finally assume $N=1$ and $I\subset \KK[x_0,x_1]$ is a nontrivial homogeneous ideal.
Thus $I=(G_1,\ldots,G_s)$ for some nonzero homogeneous generators $G_i$.
Let $F$ be the greatest common divisor of the $G_i$, and for each $i$ let $H_iF=G_i$.
Then $I=(F)Q$ where $Q=(H_1,\ldots,H_s)$. If $\deg H_i=0$ for some $i$, we have $Q=(1)$ and so $I=(F)$.
If $\deg H_i>0$ for all $i$, then $Q$ is primary for $(x_0,x_1)$.

If $I=(F)$, then $I$ is normal, so we have $\overline{I^t}=I^t$ for all $t\geq1$. 
So say $I=(F)Q$ where $Q$ is primary for $(x_0,x_1)$, hence $I^t=(F^t)Q^t$ for each $t\geq1$.

Note that $I^t=\overline{I^t}$ if and only if $Q^t=\overline{Q^t}$.
(Here's why. Assume $I^t=\overline{I^t}$.
By \cite[Remark 1.3.2(2)]{HuSw}, we have $\overline{I^t:J}\subseteq \overline{I^t}:\overline{J}$
for any ideal $J$. Take $J=(F^r)$. Then $Q^t\subseteq \overline{Q^t}=
\overline{I^t:J}\subseteq \overline{I^t}:\overline{J}=I^t:(F^r)=Q^t$.
Alternatively, say $x\in \overline{Q^t}$. 
Then $x^n+a_1x^{n-1}+\cdots+a_n=0$ for some $n$ and some $a_i\in Q^{ti}$.
Multiplying by $F^{tn}$ gives $(F^tx)^n+F^ta_1(F^tx)^{n-1}+\cdots+(F^t)^na_n=0$ so
$F^tx\in \overline{(F^t)Q^t}=\overline{I^t}=I^t=(F^t)Q^t$, so $x\in Q^t$.
Now assume $Q^t = \overline{Q^t}$. Say $x\in \overline{I^t}$.
Then $x^n+F^ta_1x^{n-1}+\cdots+(F^t)^na_n=0$ for some $n$ with $a_i\in Q^{ti}$.
Say $F^t=AP^m$ where $P$ is an irreducible factor of $F$ and $P$ does not divide $A$. Then $P$ divides $x$.
Let $yP=x$. Dividing out gives
$y^n+AP^{m-1}a_1y^{n-1}+\cdots+A^nP^{n(m-1)}a_n=0$.
We can keep dividing out until we have
$z^n+Aa_1z^{n-1}+\cdots+A^na_n=0$, where $zP^m=x$ and $z\in\overline{(A)Q^t}$.
Continuing in this way, dividing out irreducible factors of $F^t$,
we eventually see that we get an element $w$ with $wF^t=x$,
where $w\in\overline{Q^t}=Q^t$, hence $x\in (F^t)Q^t=I^t$.) 

Now we claim that $Q$ is normal if and only if $Q=\overline{Q}$.
Given this, if $I=\overline{I}$, then $Q=\overline{Q}$, hence $Q^t=\overline{Q^t}$
so $I^t=\overline{I^t}$ (i.e., $I$ is normal), which is what we needed to show.
But note that $Q$ normal implies by definition that $Q=\overline{Q}$.
Conversely, assume $Q=\overline{Q}$. It suffices to show
$Q^t=\overline{Q^t}$ for each $t$. 
By \cite[Proposition 4.8]{AM} and \cite[Proposition 1.1.4]{HuSw}, 
it is enough to check this after localizing at $(x_0,x_1)$.
But by results of Zariski, $Q=\overline{Q}$ implies 
$Q^t=\overline{Q^t}$ for each $t$ in the local case
\cite[Theorem 14.4.4]{HuSw}.
\end{proof}

\begin{theorem}\label{BSkoda}
Let $I\subset \KK[\PP^N]$ be a nontrivial homogeneous ideal, $N\geq1$.
\begin{enumerate}
\item[(a)] We have 
$$1\leq\rho_{int}(I)=\max\Big\{\big\{\frac{m}{r}: \overline{I^m}\not\subseteq I^r\big\}\cup\{1\}\Big\}.$$
\item[(b)] If $I=\overline{I}$ and $N>1$, then $\rho_{int}(I)<N$.
\item[(c)] If $I^{(m)}=\overline{I^{(m)}}$ for all $m\geq1$ (as for example is the case for 
$I=I(Z)$ for a fat point subscheme  $Z\subset\PP^N$), then
$\rho_{int}(I)\leq\rho(I)$. 
\item[(d)] If $I^{(m)}=\overline{I^m}$ for all $m\geq1$, then
$\rho_{int}(I)=\rho(I)$. 
\end{enumerate}
\end{theorem}

\begin{proof}
(a) By Lemma \ref{rhoIntLem} we have $1\leq \rho_{int}(I)$.
In order for $\rho_{int}(I) > 1$ there must be a pair of positive integers $(r_0,m_0)$
with both $m_0/r_0 > 1$ and $\overline{I^{m_0}}\not\subseteq I^{r_0}$ (and hence
$m_0<r_0+N$ by Brian{\c c}on-Skoda). Set $c=m_0/r_0$.
If $\rho_{int}(I)>c$, then as before we have $(r,m)$
with both $m/r > c$ and $\overline{I^m}\not\subseteq I^r$ (and hence $m<r+N$).
But there are only finitely many pairs $(r,m)$ with $cr<m$ and $m<r+N$
(in particular, we have $r< N/(c-1)$ and $cr<m<r+N$).
Thus either $\rho_{int}(I)=1$ or $\rho_{int}(I)=\max\{\frac{m}{r}: \overline{I^m}\not\subseteq I^r\}$,
hence $\rho_{int}(I)=\max\Big\{\big\{\frac{m}{r}: \overline{I^m}\not\subseteq I^r\big\}\cup\{1\}\Big\}.$

(b) Now assume $I=\overline{I}$ and $N>1$. By Lemma \ref{rhoIntLem} we have $\overline{I^{Nr}}\subseteq I^r$ for $r\geq1$,
so $\overline{I^m}\not\subseteq I^r$ implies $m/r<N$, hence $\rho_{int}(I)$, being either 1 or
a maximum of values $m/r$ less than $N$, is less than $N$.

(c) Here we assume $I^{(m)}=\overline{I^{(m)}}$ for all $m\geq1$. 
Then since $I^m\subseteq I^{(m)}$ we have $\overline{I^m}\subseteq I^{(m)}$,
so $\overline{I^m}\not\subseteq I^r$ implies $I^{(m)}\not\subseteq I^r$,
and hence $\rho_{int}(I)\leq\rho(I)$.

(d) Assuming $I^{m}=\overline{I^{(m)}}$ for all $m\geq1$, we have
$$\rho_{int}(I) =\sup\Big\{\frac{m}{r} : \overline{I^m}\not\subseteq I^r\Big\}=
\sup\Big\{\frac{m}{r} : I^{(m)}\not\subseteq I^r\Big\}=\rho(I).$$
\end{proof}

We do not know any examples with $\rho_{int}(I(Z))>1$. For $Z\subset\PP^2$, there are none, by the next result,
which is an immediate consequence of \cite[Theorem 3.3]{AH1}.
We thank Huneke for alerting us to this result.

\begin{corollary}\label{CorIntRho3}
Let $Z\subset\PP^N$ be a nontrivial fat point subscheme and let $I=I(Z)$.
Then we have the following:
\begin{enumerate}
\item[(a)] $\overline{I^{N+m-1}}\subseteq I^m$ for $m\geq1$;
\item[(b)] $\rho_{int}(I(Z))\leq N/2$ for $N\geq2$; and
\item[(c)] $\rho_{int}(I(Z))=1$ for $N=1,2$.
\end{enumerate}
\end{corollary}

\begin{proof}
(a) Let $M=(x_0,\ldots,x_N)\subset \KK[\PP^N]$.
Using $I$ as a reduction for $I$ and $\ell$ for the analytic spread of $I$,
\cite[Theorem 3.3]{AH1} states (after localizing at the irrelevant ideal $M$)
that $\overline{I^{\ell+m}}\subseteq I^{\ell-N+m+1}$ for $m\geq0$.
But $N\leq \ell\leq N+1$, so for $\ell=N$ we have
$\overline{I^{N+m}}\subseteq I^{m+1}$, while for $\ell=N+1$ we have
$\overline{I^{N+m+1}}\subseteq I^{m+2}$, both for $m\geq0$. Either way, we have
$\overline{I^{N+m-1}}\subseteq I^m$ for $m\geq1$, and hence 
$\overline{I^{N+m-1}}\subseteq I^m$ holds without localizing, since all of the ideals
are homogeneous.

(b) Assume $N\geq 2$. For $r=1$ we have $\overline{I^m}\not\subseteq I^r$ for all $m\geq 1$,
so consider $r>1$. Then we have $\overline{I^m}\subseteq I^r$ for all $m\geq N+r-1$,
so the fractions $m/r$ for which we have $\overline{I^m}\not\subseteq I^r$ 
are contained in the set $\{m/r : 1\leq m\leq N+r-2, r\geq2\}$. The supremum occurs for
$m=N+r-2$ and $r=2$, hence the supremum is $N/2$, so $\rho_{int}(I(Z))\leq N/2$.

(c) When $N=1$ we have $\rho_{int}(I(Z))=1$ from Lemma \ref{rhoIntLem}, and
when $N=2$ we have $\rho_{int}(I(Z))=1$ from (b).
\end{proof}

If $\rho_{int}(I(Z))$ is always 1, then the next result would imply that 
$\widehat{\rho}(I(Z))=1$ if and only if $\rho(I(Z))=1$.
In particular, it shows that $\widehat{\rho}(I(Z))=1$ if and only if $\rho(I(Z))=1$ for every fat point subscheme $Z\subset\PP^2$.

\begin{corollary}\label{CorIntRho1}
Let $Z\subset\PP^N$ be a nontrivial fat point subscheme.
If $\widehat{\rho}(I(Z))=1$, then $\rho_{int}(I(Z))=\rho(I(Z))$,
hence $\widehat{\rho}(I(Z))=1$ if and only if $\rho(I(Z))=1$ when $N=2$.
\end{corollary}

\begin{proof}
By Theorem \ref{MT3}, $\widehat{\rho}(I(Z))=1$ implies $I(mZ) = \overline{I(Z)^m}$,
so $\rho_{int}(I(Z))=\rho(I(Z))$. Since $\rho_{int}(I(Z))=1$ when $N=2$ by Corollary \ref{CorIntRho3},
and since $\rho(I(Z))=1$ implies $\widehat{\rho}(I(Z))=1$, we have
$\widehat{\rho}(I(Z))=1$ if and only if $\rho(I(Z))=1$ when $N=2$.
\end{proof}

We now recover a version of \cite[Corollary 4.17]{DFMS}.

\begin{corollary}\label{CorIntRho2}
Let $Z\subset\PP^N$ be a nontrivial fat point subscheme.
Then $\rho(I(Z))=1$ if and only if $\widehat{\rho}(I(Z))=\rho_{int}(I(Z))=1$.
\end{corollary}

\begin{proof}
The proof is immediate from 
$$1\leq\rho_{int}(I(Z))\leq\rho(I(Z)),$$
$$1\leq\widehat{\rho}(I(Z))\leq\rho(I(Z))$$
and Corollary \ref{CorIntRho1}.
\end{proof}

\begin{remark}\label{rhointRem}
For a nontrivial homogeneous ideal $I\subset\KK[\PP^N]$, it is also of interest to define
$$\rho^{int}(I)=\sup\Big\{\frac{m+1}{r} : \overline{I^m}\not\subseteq I^r\Big\}.$$
This is exactly what \cite{DFMS} denotes as $K(I)$. 
Clearly we have $\rho_{int}(I)\leq \rho^{int}(I)$, and by applying Theorem \ref{BSkoda}
we see we have equality if and only if $\rho^{int}(I)=1$, in which case $I$ is normal. By
\cite[Proposition 4.19]{DFMS} we have $\rho(I)\leq \widehat{\rho}(I)\rho^{int}(I)$.
Thus we have 
$$\rho_{int}(I(Z))\leq \rho(I(Z))\leq \widehat{\rho}(I(Z))\rho^{int}(I(Z))$$
for every fat point subscheme $Z\subset\PP^N$.
\end{remark}

\section{Grifo's Conjecture}
We now discuss Grifo's containment conjecture \cite[Conjecture 2.1]{G}. 
In our context it says the following
(we note it is true and easy to prove for $N=1$).

\begin{conjecture}\label{GrifoConj}
Let $I\subseteq\KK[\PP^N]$ be a radical homogeneous ideal.
Then $I^{(h_Ir-h_I+1)}\subseteq I^r$ for all $r\gg0$.
\end{conjecture}

Remark 2.7 of \cite{G} shows the conjecture holds for $I$ if $\rho(I)<h_I$
(whether $I$ is radical or not), or if $\rho'(I)<h_I$,
and raises the question of whether it holds when $\widehat{\rho}(I)<h_I$.
This was answered affirmatively by \cite[Proposition 2.11]{GHM} with a direct proof
(we thank Grifo and Huneke for bringing this result to our attention).
Our results also answer this question affirmatively, in two ways.
By Theorem \ref{MT1}, $\widehat{\rho}(I)<h_I$ implies $\rho(I)'<h_I$,
hence Conjecture \ref{GrifoConj} holds for $I$ by the results of \cite{G}.
And by Theorem \ref{MT2}, $\widehat{\rho}(I)<h_I$ implies $\rho(I)<h_I$,
so again Conjecture \ref{GrifoConj} holds for $I$ by the results of \cite{G}.

\begin{remark}\label{GrifoCorRem}
When $I=I(Z)$ for a fat point scheme $Z\subset\PP^N$ we have $h_I=N$.
No examples of a fat point scheme $Z\subset\PP^N$ (radical or not) are known
for which it is not true that $I(Z)^{(Nr-N+1)}\subseteq I(Z)^r$ for all $r\gg0$.
By Remark 2.7 of \cite{G}, one approach to proving that $I(Z)^{(Nr-N+1)}\subseteq I(Z)^r$ for all $r\gg0$
holds for all $Z$ is to show $\rho(I(Z))<N$ whenever $N>1$. 
This raises the question of: for which $Z$ is it known that $\rho(I(Z))<N$?

Let $Z=m_1p_1+\cdots+m_sp_s$. If ${\rm gcd}(m_1,\ldots,m_s)>1$,
then $\rho(I(Z))<N$ by \cite[Proposition 2.1(2)]{TX}.
Thus it is the cases with ${\rm gcd}(m_1,\ldots,m_s)=1$ that remain of interest.

Another approach is to apply our Theorem \ref{MT2}: $\rho(I(Z))<N$ holds
if $\widehat{\rho}(I(Z))<N$.
Aiming to show $\widehat{\rho}(I(Z))<N$ has the advantage that the results of \cite{DFMS,DD} suggest
that $\widehat{\rho}(I(Z))$ is more accessible computationally than is $\rho(I(Z))$.
Another advantage is that in most cases where $\widehat{\rho}(I(Z))$ is known we have
$\widehat{\rho}(I(Z))=\frac{\alpha(I(Z))}{\widehat{\alpha}(I(Z))}$
(but see Example \ref{Ex3})
and in all known cases we have $\widehat{\alpha}(I(Z))\geq \frac{\alpha(I(mZ))+N-1}{m+N-1}$ for all $m\geq 1$.
Assuming both we have
$$\widehat{\rho}(I(Z))=\frac{\alpha(I(Z))}{\widehat{\alpha}(I(Z))}\leq \frac{\alpha(I(Z))}{\frac{\alpha(I(Z))+N-1}{N}}
= N\frac{\alpha(I(Z))}{\alpha(I(Z))+N-1}<N$$
and thus we would have $\rho(I(Z))<N$.
\end{remark}

The previous paragraph merits further discussion.
First we recall \cite[Conjecture 2.1]{HaHu}, which if true would refine the containment
$I(rNZ)\subseteq I(Z)^r$ of \cite{ELS, HoHu}:

\begin{conjecture}\label{HaHu}
Let $Z\subset\PP^N$ be a fat point scheme. Let $M=(x_0,\ldots,x_N)$. Then
$$I(rNZ)\subseteq M^{r(N-1)}I(Z)^r$$ 
holds for all $r>0$.
\end{conjecture}

A further refinement of $I(rNZ)\subseteq I(Z)^r$ is that
$I(r(m+N-1)Z)\subseteq I(mZ)^r$ \cite{ELS,HoHu}. This suggests
a refinement of Conjecture \ref{HaHu} (cf. \cite[Question 4.2.3]{HaHu}), namely
\begin{equation}\label{HaHuRefinement}
I(r(m+N-1)Z)\subseteq M^{r(N-1)}I(mZ)^r.
\end{equation}

If \eqref{HaHuRefinement} were true, then the following conjecture would also be true:

\begin{conjecture}\label{Chud}
Let $Z\subset\PP^N$ be a fat point scheme. Then
$$\widehat{\alpha}(I(Z))\geq \frac{\alpha(I(mZ))+N-1}{m+N-1}$$ 
for all $m\geq 1$.
\end{conjecture}

The proof that Conjecture \ref{HaHu} implies Conjecture \ref{Chud} when $m=1$ is 
given in \cite{HaHu}. The same argument shows that \eqref{HaHuRefinement}
implies Conjecture \ref{Chud}.

The first version of Conjecture \ref{Chud} was posed by Chudnovsky \cite{Ch} 
over the complex numbers for the case that $m=1$ and $Z$ is reduced.
He also sketched a proof of his conjecture for $N=2$ which works over any algebraically closed field
(see \cite{HaHu} for a proof).
Conjecture \ref{Chud} assuming $Z$ reduced was posed by Demailly \cite{Dem}.

The best result currently known is by Esnault and Viehweg \cite{EV}. It is over the complex numbers
and says that if $Z$ is reduced with $N>1$, then
$$\widehat{\alpha}(I(Z))\geq \frac{\alpha(I(mZ))+1}{m+N-1}$$
holds for all $m\geq 1$.
Thus for reduced $Z$ over the complex numbers, taking $m=1$, we have
$$\frac{\alpha(I(Z))}{\widehat{\alpha}(I(Z))}\leq \frac{\alpha(I(Z))}{\frac{\alpha(I(Z))+1}{N}}
= N\frac{\alpha(I(Z))}{\alpha(I(Z))+1}<N.$$
Hence $\rho(I(Z))<N$ holds over the complex numbers 
whenever $Z$ is reduced and $\widehat{\rho}(I(Z)) = \frac{\alpha(I(Z))}{\widehat{\alpha}(I(Z))}$.

\section{Examples and questions}\label{ExSect}
For this section assume $Z$ is a fat point subscheme of $\PP^N$ 
(so $I(Z)$ is a nontrivial ideal of $\KK[\PP^N]$).

\begin{example}\label{Ex1}
If $I(Z)$ is a complete intersection or a power thereof, then 
$I(mZ)=I(Z)^m$ for all $m\geq1$, hence $\rho(I(Z))=1$ by Theorem \ref{MT3}.
Again by Theorem \ref{MT3}, if $Z$ is reduced, then 
$\rho(I(Z))=1$ if and only if $I(Z)$ is a complete intersection.
However, when $Z$ is not reduced, $\rho(I(Z))=1$ does not imply
$I(Z)$ is a complete intersection, or even a power of a complete intersection
ideal (homogeneous or not),
as is shown by Example \ref{verticesEx} and Remark \ref{FatPtsAndCI1}.
(Additionally, let $p_1,p_2,p_3\in\PP^N$ be noncollinear points and let
$Z=m_1p_1+m_2p_2+m_3p_3$ with $1\leq m_1\leq m_2\leq m_3$.
If either $m_1+m_2\leq m_3$ or if $m_1+m_2+m_3$ is even,
then by \cite[Theorem 2]{BZ} we have
$I(mZ)=I(Z)^m$ for all $m\geq1$ and hence $\rho(I(Z))=1$.
See \cite[Example 5.1]{BH2} for additional examples in $\PP^2$ of $Z$
for which all powers of $I(Z)$ are symbolic.
By Remark \ref{FatPtsAndCI2}, in none of these cases
is there a homogeneous complete intersection ideal $J$ 
such that $I(Z)=J^r$ for some $r\geq1$.
See \cite[Proposition 3.5]{HaHu} for a criterion for $Z\subset\PP^2$ 
such that $I(mZ)=I(Z)^m$ for all $m\geq1$; this gives further examples
for which $I(Z)$ is not a power of a homogeneous complete intersection.)
\end{example}

It is an interesting problem to clarify which fat point subschemes $Z\subset\PP^N$ 
have $I(mZ)=I(Z)^m$ for all $m\geq1$. 
We do not know any $Z$ for which the analytic spread $\ell(I(Z))=N$ but for which
$I(mZ)=I(Z)^m$ fails for some $m\geq1$.
Nor do we know any $Z$ for which either $\rho(I(Z))=1$ or $\widehat{\rho}(I(Z))=1$ but for which
$I(mZ)=I(Z)^m$ fails for some $m\geq1$.
Thus we have the following question.

\begin{question}\label{Q-NonreducedZCI}
Is it true that all powers of $I(Z)$ 
are symbolic (i.e., $I(mZ)=I(Z)^m$ for all $m\geq1$)
if the analytic spread of $I(Z)$ is $N$, or if
$\rho(I(Z))=1$ or $\widehat{\rho}(I(Z))=1$? 
\end{question}

\begin{example}\label{Ex1b}
It is worth noting here that there is a monomial ideal $I$ with $\rho(I)=\widehat{\rho}(I)=1$ 
where $I^m\subsetneq I^{(m)}$ for every $m>1$ 
(\cite[Remark 3.5]{DD}; we thank DiPasquale and Grifo for bringing this to our attention).
Since the ideal $I$ given in \cite[Remark 3.5]{DD} is a radical monomial ideal and powers of monomial primes are primary,
we have $I^{(m)}=\cap_{P\in{\rm Ass(I)}}P^m$. But monomial primes are normal,
so $I^{(m)}$ is integrally closed. Thus $\overline{I^m}\subseteq I^{(m)}$, and since
$\widehat{\rho}(I)=1$ we have by \cite[Corollary 4.16]{DFMS} that 
$I^{(m)}\subseteq \overline{I^m}$, so $I^{(m)}=\overline{I^m}$ for all $m\geq 1$.
Thus $I^m$ is integrally closed only for $m=1$, but nonetheless $\rho_{int}(I)=\rho(I)=1$.
\end{example}

Answering Question \ref{Q-NonreducedZCI} is closely related to whether
having $\rho(I(Z))=1$ or $\widehat{\rho}(I(Z))=1$ gives a 
complete solution to the containment problem for $I(Z)$.
If $\rho(I(Z))=1$ or $\widehat{\rho}(I(Z))=1$,
then by Theorem \ref{MT3} when $I(Z)$ is radical, we do have a complete solution to the containment 
problem for $I(Z)$: $I(mZ)\not\subseteq I(Z)^r$ for $m<r$ and otherwise we have 
$I(mZ)\subseteq I(Z)^r$.
When $I(Z)$ is not radical, we do not know if having either $\rho(I(Z))=1$ or $\widehat{\rho}(I(Z))=1$
solves the containment problem for $I(Z)$. For example, having $\rho(I(Z))=1$ means
$\widehat{\rho}(I(Z))=1$ and it means $I(mZ)\not\subseteq I(Z)^r$ for $m<r$ and 
$I(mZ)\subseteq I(Z)^r$ for $m>r$, but we do not know for which $m\geq1$ that we have 
$I(mZ)\subseteq I(Z)^m$.

We also do not know any examples with $\widehat{\rho}(I(Z))=1$ but $\rho(I(Z))>1$.
This raises the following question.

\begin{question}\label{Q4b}
Does $\widehat{\rho}(I(Z))=1$ always imply $\rho(I(Z))=1$?
\end{question}

\begin{example}\label{Ex4}
The assumption $I=\overline{I}$ is needed in both Lemma \ref{rhoIntLem} and Theorem \ref{BSkoda}(b).
For example, take $t=1$ and any $N\geq1$.
Let $I=(x_0^{N+1},x_1^{N+1},\ldots,x_N^{N+1})\subset\KK[\PP^N]$.
Then $x_0^N\cdots x_N^N\in\overline{I^N}=(x_0,\ldots,x_N)^{N(N+1)}$ but 
$x_0^N\cdots x_N^N\not\in I$, so $\overline{I^{tN}}\not\subseteq I^t$ and $\rho_{int}(I)=N$.
\end{example}

Although $I(Z)$ is not always normal
(see Example \ref{Ex2}) and hence $I(Z)^m=\overline{I(Z)^m}$ can fail, 
we do not have an example with $\overline{I(Z)^m}\not\subseteq I(Z)^r$
when $m>r$, so we do not know of any $Z$ for which $\rho_{int}(I(Z))\neq1$. 
If $\rho_{int}(I(Z))=1$ were always true, then 
by Corollary \ref{CorIntRho2} we would have that $\rho(I(Z))=1$
if and only if $\widehat{\rho}(I(Z))=1$, thus answering Question \ref{Q4b}.
This raises the following question. 

\begin{question}\label{FinalQuest2}
Is it ever true that
$\rho_{int}(I(Z)) >1$?
\end{question}

The next example shows that $\rho_{int}(I(Z))=1$ does not force $\rho(I(Z))=1$ or $\widehat{\rho}(I(Z))=1$,
even if $Z$ is reduced. In contrast, 
we know that (b) (and hence (a)) of Theorem \ref{MT3} implies
$\rho_{int}(I(Z))=1$, but we do not know if any of the other criteria of 
Theorem \ref{MT3} imply $\rho_{int}(I(Z))=1$, unless $Z$ is reduced or $N=2$.

\begin{example}\label{Ex5}
Examples occur with $\rho_{int}(I(Z))=1$ but with $\widehat{\rho}(I(Z))>1$.
Let $Z\subset\PP^N$ be a {\it star configuration}, meaning we have 
$s>N$ general hyperplanes, and $Z$ is the reduced scheme
consisting of $\binom{s}{N}$ points, where each point is the intersection of $N$ of 
the $s$ hyperplanes; see \cite[Definition 3.8]{HaHu}.
Then the ideal $I(Z)$ has $\alpha(I(Z))={\rm reg}(I(Z))=s+N-1$ \cite[Lemma 2.4.2]{BH1}
and $\widehat{\alpha}(I(Z))=s/N$ \cite[Lemma 2.4.1]{BH1}.
As noted in \S\ref{bkgrnd}, $\alpha(I(Z))={\rm reg}(I(Z))$ 
implies that $I(Z)$ is normal and hence $\rho_{int}(I(Z))=1$,
but $\alpha(I(Z))={\rm reg}(I(Z))$ also implies that 
$\frac{\alpha(I(Z))}{\widehat{\alpha}(I(Z))}=\widehat{\rho}(I(Z))=\rho(I(Z))$,
and in the case of a star configuration we have 
$\frac{\alpha(I(Z))}{\widehat{\alpha}(I(Z))}=N(s-N+1)/s>1$ when $N>1$.
However we do not know of any $Z$ with
$\widehat{\rho}(I(Z))<\rho_{int}(I(Z))$.
\end{example}

\begin{question}\label{FinalQuest}
Is it ever true that
$\rho_{int}(I(Z)) > \widehat{\rho}(I(Z))$?
Do any of (c), (d) or (e) of Theorem \ref{MT3} imply $\rho_{int}(I(Z))=1$, when $Z$ is not reduced and $N>2$?
\end{question}

\begin{example}\label{Ex2}
When $\widehat{\rho}(I(Z)) >1$, it is known that $\widehat{\rho}(I(Z)) < \rho(I(Z))$ can occur.
For example, let $I = (x(y^n-z^n),y(z^n-x^n),z(x^n-y^n))\subset \CC[x,y,z]=\CC[\PP^2]$,
so $I=I(Z)$ where $Z$ is a certain set of $n^2+3$ points.
Then by \cite[Theorem 2.1]{DHNSST}, we have $\widehat{\rho}(I(Z)) =
\frac{\alpha(I(Z))}{\widehat{\alpha}(I(Z))} =(n+1)/n < 3/2=\rho(I(Z))$ for $n\geq3$.
Moreover, since $\widehat{\rho}(I(Z)) < \rho(I(Z))$, it follows that
$I(Z)$ cannot be normal. 
\end{example}

\begin{example}\label{Ex3}
Moreover, examples of $Z$ occur with $\frac{\alpha(I(Z))}{\widehat{\alpha}(I(Z))} < \widehat{\rho}(I(Z))$. 
The results of \cite{DFMS} suggest that this should occur, but up to now
no explicit examples have been given.
For one such explicit example (indeed, the first we are aware of), let 
$Z$ consist of 8 points in the plane, where 3 of the points (say $p_1,p_2,p_3$)
are general (we may as well assume they are the coordinate vertices)
and the other 5 (say $p_4,\ldots,p_8$) are
on a general line $L$ (defined by a linear form $F$) and are general on that line. 
Then one can show that $\alpha(I(Z))=3$, $\widehat{\alpha}(I(Z))=5/2$
and $I(25sZ)\subsetneq I(Z)^{19s+1}$ for all $s\geq1$, and hence that 
$\alpha(I(Z))/\widehat{\alpha}(I(Z)) =6/5 < 25/19\leq \widehat{\rho}(I(Z))$.
(We now sketch the justification of these claims. 
The key is that $(I(25sZ))_{65s}$ vanishes on $L$ with order $15s$,
but $(I(Z)^{19s+1})_{65s}$ vanishes on $L$ with order $15s+3$, hence 
$(I(25sZ))_{65s}\not\subseteq (I(Z)^{19s+1})_{65s}$ and so $I(25sZ)\not\subseteq I(Z)^{19s+1}$.
It is easy to check that $\alpha(I(Z))=3$.
Using Bezout's Theorem, one can show that 
$(I(2tZ))_{5t}=(xyzF^2)^t\KK\subset\KK[\PP^2]=\KK[x,y,z]$
and hence (since $\dim ((I(mZ))_i)>0$ implies $\dim ((I(mZ))_j)>1$ for all $j>i$)
that $\alpha(I(2tZ))=5t$.
(In more detail, note that $(xyzF^2)^t\in (I(2tZ))_{5t}$.
If $(I(2tZ))_{5t}$ contained anything more than scalar multiples of
$(xyzF^2)^t$, then $(I(8tZ))_{20t}$ would contain more than scalar multiples of
$(xyzF^2)^{4t}$, so it's enough to consider $(I(8tZ))_{20t}$.
But by Bezout, we have 
\begin{align*}
(I(8tZ))_{20t}&=F^{5t}((I(8t(p_1+p_2+p_3)+3t(p_4+\cdots+p_8)))_{15t})\\
&=(xyz)^tF^{5t}((I(3t(2(p_1+p_2+p_3)+(p_4+\cdots+p_8))))_{12t})\\
&=(xyz)^tF^{5t+1}((I((3t(2(p_1+p_2+p_3))+(3t-1)(p_4+\cdots+p_8))))_{12t-1})\\
&=(xyz)^{t+1}F^{5t+1}((I((3t-1)(2(p_1+p_2+p_3)+(p_4+\cdots+p_8))))_{12t-4})\\
&=\cdots=(xyz)^{4t}F^{8t}\KK.)\\
\end{align*}
\vskip-\baselineskip
\noindent It follows that $\widehat{\alpha}(I(Z))=5/2$.
We also have by Bezout that $(I(25sZ))_{65s}=F^{15s}(I(5sZ'))_{50s}$, where
$Z'$ is the fat point subscheme obtained by taking the three coordinate vertices
with multiplicity 5 and the five points on $L$ with multiplicity 2. Moreover, one can
check that $(I(Z'))_{10}$ has greatest common divisor 1 (i.e., it defines a linear system
which is fixed component free); to see this, write elements of $(I(Z'))_{10}$ in two different ways
using conics and lines, where the two ways have no factors in common.
Thus $(I(5sZ'))_{50s}$ also defines a linear system that is fixed component free,
so $(I(25sZ))_{65s}=F^{15s}(I(5sZ'))_{50s}$ has $15sL$ as the divisorial
part of its base locus. But $(I(Z))_t$ has $L$ in its base locus for $t=3,4$
and is fixed component free for $t\geq5$. Now $(I(Z))^{19s+1}$ is spanned by 
products of $a$ homogeneous elements of $I(Z)$ of degree 5 or more and $b$ elements
of degree 4 or less with $a+b=19s+1$. Such an element vanishes on $L$ to order at least $b$.
To minimize $b$ we want $a$ as large as possible, and hence we want
to take all $a$ elements to have degree 5 and as many of the $b$ elements as possible
to have degree 3, and thus we have the inequality $5a+3b\leq 65s$. Solving
$5a+3b\leq 65s$ given $a+b=19s+1$ gives $2a\leq 8s-3$ and hence,
taking $a$ as large as possible we have $a=4s-2$, so $b=15s+3$.
I.e., every element of $((I(Z))^{19s+1})_{65s}$ vanishes on 
$L$ to order at least $15s+3$.) We do not know if
$\widehat{\rho}(I(Z))=25/19$, nor do we know the value of $\rho(I(Z))$.
\end{example}

\begin{question}\label{rhohatValue}
For the example $Z$ in the previous paragraph, what are the exact values of 
$\widehat{\rho}(I(Z))$, $\rho(I(Z))$ and $\rho_{int}(I(Z))$?
\end{question}

As Remark \ref{GrifoCorRem} explains, if the next question
has a negative answer, then Conjecture \ref{GrifoConj} holds
for $I=I(Z)$ whenever $\widehat{\rho}(I(Z))=\alpha(I(Z))/\widehat{\alpha}(I(Z))$.

\begin{question} 
Is it ever false that
$$\widehat{\alpha}(I(Z))\geq \frac{\alpha(I(Z))+1}{N}?$$
\end{question}

By \cite{DFMS}, $\widehat{\rho}(I(Z))$ is equal to the maximum
of $v(I(Z))/\widehat{v}(I(Z))$ for valuations $v$ supported on $I(Z)$, and this maximum always
occurs for what is known as a Rees valuation. Thus we have
$$\frac{v(I(Z))}{\widehat{v}(I(Z))}\leq \widehat{\rho}(I(Z))\leq N,$$ 
hence
$$\frac{v(I(Z))}{N}\leq \widehat{v}(I(Z)).$$ 
This raises the question of whether Chudnovsky-like bounds occur for valuations:

\begin{question}\label{FinalQuest3}
Is it always true for $N>1$ that
$$\frac{v(I(Z))+1}{N}\leq \widehat{v}(I(Z))?$$ 
\end{question}

If the answer is affirmative, then for some $v$ we would have 
$$\widehat{\rho}(I(Z))=\frac{v(I(Z))}{\widehat{v}(I(Z))} \leq \frac{v(I(Z))}{\frac{v(I(Z))+1}{N}}<N$$
which would confirm Grifo's Conjecture for $I(Z)$.

\begin{example}\label{verticesEx}
Consider $Z=p_1+\cdots+p_N+2p_{N+1}\subset\PP^N$ where the points $p_i$ are 
the coordinate vertices. Here we show that $I(Z)^m=I(mZ)$ for all $m\geq1$.
For $N=1$, $I(Z)^m=I(mZ)$ holds for every $Z$ (since fat point ideals are principal), so assume $N>1$.
(We note that the case $N=2$ is covered by the results of  \cite{BZ}.)
Since $I(Z)$ is a monomial ideal, we merely have to check for each monomial 
$f=x_0^{a_0}\cdots x_N^{a_N}\in I(mZ)$ that $f\in I(Z)^m$; i.e., that $f$ can 
is divisible by a product of $m$ monomials, each in $I(Z)$.
But if we choose coordinates $x_i$ such that $x_j\neq0$ at $p_j$ for each $j$,
then $f\in I(mZ)$ is equivalent to $a_1+\cdots+a_N\geq 2m$ and $a_0+a_1+\cdots+a_N-a_i\geq m$
for each $0<i\leq N$. Without loss of generality we may assume
$a_1\leq a_2\leq \cdots\leq a_N$, so the inequalities $a_0+a_1+\cdots+a_N-a_i\geq m$ reduce to the single
inequality $a_0+a_1+\cdots+a_{N-1}\geq m$. Let us set $b=a_1+\cdots+a_{N-1}$.
Then $f\in I(mZ)$ if and only if $a_0+b\geq m$ and $b+a_N\geq 2m$,
so we want to show $a_0+b\geq m$ and $b+a_N\geq 2m$ implies $f\in I(Z)^m$.

First suppose $b\geq m$. We have two cases: $b+a_N=2m$ and $b+a_N>2m$.
First assume $b+a_N=2m$, so $0<b/(N-1)\leq a_{N-1}\leq a_N\leq b$. Let 
$e_0=(e_{01},\ldots,e_{0N})=(a_1,\cdots,a_N)$ so $e_0$ is the exponent vector of 
$g_0=x_1^{e_{01}}\cdots x_N^{e_{0N}}=x_1^{a_1}\cdots x_N^{a_N}$.
Let $b_0=b=e_{01}+\cdots+e_{0,N-1}$.
Note that $g_0\in I(mZ)$ and $g_0$ divides $f$, so it's enough to show $g_0\in I(Z)^m$. 

Starting with $e_0$ and $b_0$, we will recursively define a sequence $e_i$ of exponent vectors 
$e_i=(e_{i1},\ldots,e_{iN})$ and nonnegative integers $b_i$ for $0\leq i\leq \omega$ with 
$b_i>0$ for $0\leq i<\omega$ and $b_\omega=0$,
satisfying the following conditions:
\begin{enumerate}
\item[(i)] the entries of $e_i$ are nonnegative and nondecreasing;
\item[(ii)] $b_i+e_{i,N}$ is even, where $b_i=e_{i1}+\cdots+e_{i,N-1}$;
\item[(iii)] $b_i\geq e_{iN}$.
\end{enumerate}

For $e_0$ we have $b_0=b\geq m>0$.
Moreover, conditions (i)-(iii) hold for $e_0$: 
(i) holds by assumption; 
(ii) holds since $b_0+e_{0,N}=b+a_N=2m$ by assumption;
(iii) holds since we noted above that $a_N\leq b$, but $e_{0N}=a_N$ and $b_0=b$.

Given that conditions (i)-(iii) hold for some $e_i$ with $b_i>0$, we now
define $e_{i+1}=(e_{i+1,1},\ldots,e_{i+1,N})$ with $b_i>b_{i+1}$.
In brief, $j_i$ and $k_i$ are chosen to be as large as possible such that $j_i<k_i$
and so that the entries of $e_{i+1}=(e_{i+1,1},\ldots,e_{i+1,N})$ are nondecreasing,
where $e_{i+1,l}=e_{il}$ for all $l$ except that $e_{i+1,j_i}=e_{ij_i}-1$ and $e_{i+1,k_i}=e_{ik_i}-1$.
More precisely, $j_i$ is the least index $t$ such that $e_{it}=e_{i,N-1}$.
Then $k_i=N$ if $e_{iN}>e_{i,N-1}$ and $k_i=j_i+1$ if $e_{iN}=e_{i,N-1}$.
Since $b_{i+1}$ is either $b_i-1$ or $b_i-2$ we have $b_i>b_{i+1}$. 

It is easy to check that the construction ensures that the
entries of $e_{i+1}$ are nonnegative and nondecreasing, so (i) holds for $e_{i+1}$.
Since $b_i+e_{i,N}$ is even and $b_{i+1}+e_{i+1,N}=b_i+e_{i,N}-2$, (ii) holds.
For (iii), if $e_{i+1,N}=e_{iN}$, then $k_i<N$ so $0<e_{i,N-2} = e_{i,N-1}=e_{i,N}$,
and we cannot have $e_{i,N-2}=1$ with $e_{i,N-3}=0$ since then the sum of the entries of
$e_i$ would be odd. Thus either $e_{i,N-2}>1$ or $e_{i,N-3}>0$; either way
$b_i\geq 2+e_{i,N-1}=2+e_{iN}$, so $b_{i+1}=b_i-2\geq e_{iN}=e_{i+1,N}$.
If $e_{i+1,N}=e_{iN}-1$, then $b_{i+1}=b_i-1$, so $b_i\geq e_{iN}$ implies
$b_{i+1}\geq e_{i+1,N}$ (and hence $b_{i+1}\geq0$).

Let $g_i$ be the monomial whose
exponent vector is $e_i$, so $g_i=x_1^{e_{i1}}\cdots x_N^{e_{iN}}$.
We have $g_\omega=1$, since $0=b_\omega\geq e_{\omega N}$,
so $e_\omega=(0,\ldots,0)$.
If $i<\omega$, then $g_i=g_{i+1}x_{j_i}x_{k_i}$.
Thus $g_0=(x_{j_0}x_{k_0})\cdots (x_{j_{\omega-1}}x_{k_{\omega-1}})$,
but each factor $x_{j_i}x_{k_i}$ is in $I(Z)$, and since
$\deg g_0=b_0+e_{0N}=2m$, we see that there are $m$ factors, so
$g_0\in I(Z)^m$, as we wanted to show.

We now consider case 2, $b+a_N>2m$ (still under the assumption that $b\geq m$). 
Starting with $(a_1,\ldots,a_{N-1},a_N)$, 
replace $a_N$ by the smallest integer $a$ satisfying $a\geq a_{N-1}$ and 
$b+a\geq 2m$. We then have that $g=x_0^{a_0}\cdots x_{N-1}^{a_{N-1}}x_N^a$
divides $f$ and is still in $I(mZ)$. Thus it is enough to show that $g\in I(Z)^m$.
If $b+a=2m$, we are done by case 1, so assume $b+a>2m$. 
Then by construction we have $a=a_{N-1}$. Since $b\geq m$,
we must have $a\leq m$. It is now not hard to find new exponents
$0\leq a_1'\leq a_2'\leq \cdots\leq a_{N-1}'$ such that 
$a_i'\leq a_i$ and $b'+a=2m$, where $b'=a_1'+\cdots+a_{N-1}'$.
Let $g_0=x_1^{a_1'}\cdots x_{N-1}^{a_{N-1}'}x_N^a$; then $g$ divides $f$
and by case 1 we have $g\in I(Z)^m$.

We are left with considering the case that $b<m$. We have
$a_0+b\geq m$ and $b+a_N\geq 2m$. Clearly we can reduce
$a_0$ and $a_N$ so that $a_0+b = m$ and $b+a_N = 2m$
(since the associated monomial $g$ divides $f$ and still is in $I(mZ)$).
So we may assume $a_0+b = m$ and $b+a_N = 2m$, in which case
$f=(x_1^{a_1}\cdots x_{N-1}^{a_{N-1}}x_N^b)x_0^{a_0}x_N^{a_N-b}
=(x_1^{a_1}\cdots x_{N-1}^{a_{N-1}}x_N^b)x_0^{m-b}x_N^{2(m-b)}
=(x_1^{a_1}\cdots x_{N-1}^{a_{N-1}}x_N^b)(x_0x_N^2)^{m-b}$.
But $x_0X_N^2\in I(Z)$ and $x_1^{a_1}\cdots x_{N-1}^{a_{N-1}}x_N^b$
is a product of $b$ factors of the form $x_ix_N$ for $1\leq i\leq N-1$,
and each of these is in $I(Z)$. Thus $f$ is a product of $m=b+(m-b)$
elements of $I(Z)$, hence $f\in I(Z)^m$.
\end{example}

\begin{remark}\label{FatPtsAndCI1}
Again consider $Z=p_1+\cdots+p_N+2p_{N+1}\subset\PP^N$ where the points $p_i$ are 
the coordinate vertices and $N>1$. Then there is no $N$-generated ideal 
$J$, homogeneous or not, such that $I(Z)=J^s$ for some $s\geq1$. 
Suppose there were; say $J=(F_1,\dots,F_N)$. 
The ideal $I(Z)$ defines the $N+1$ coordinate lines
in affine $N+1$ space, all taken with multiplicity 1 
except one taken with multiplicity 2.
Since each $F_i$ vanishes on each coordinate line, none 
of the $F_i$ can have any terms which are a power of
a single variable (i.e., every term involves a product of two or more variables).
Therefore, none of the $F_i$ have terms of degree less than 2. 
Since $(F_1,\dots,F_N)^s=J^s=I(Z)$ and since $I(Z)$ has elements with terms of degree 2,
we see that $s=1$ (otherwise the least degree of a term of an element of $J^s$
would be at least $2s\geq4$). But $s=1$ implies that $J=I(Z)$ is homogeneous,
and an easy argument shows that the least number of homogeneous generators
of a homogeneous ideal is the least number of generators possible,
which in this case is $\binom{N}{2}+N>N$ (since, for a particular ordering of the points $p_i$
we have $I(Z)=(x_ix_j:i>0, j>0)+(x_0x_i^2:i>0)$).
\end{remark}

\begin{remark}\label{FatPtsAndCI2}
Let $Z=m_1p_1+\cdots+m_rp_r\subset\PP^N$ with $N>1$, $m_i>0$ for all $i$ and
$\KK[\PP^N]=\KK[x_0,\ldots,x_N]$.
Here we show that there is an ideal $J$ generated by $N$ forms such that
$I(Z)=J^m$ for some $m\geq1$
if and only if $m_1=\cdots=m_r=m$ where $I(p_1+\cdots+p_r)$
is generated by $N$ forms. The reverse implication is known \cite[Lemma 5, Appendix 6]{ZS},
so assume $I(Z)=J^m$ for some $m\geq1$ with forms $F_i$ such that
$J=(F_1,\ldots,F_N)$.
Choose coordinates $x_0,\ldots,x_N$ such that none of the points lies on $x_0=0$
and $p_1=(1,0,\ldots,0)$.
Let $U_0$ be the affine neighborhood defined by $x_0\neq0$,
and let $f_i(x_1,\ldots,x_N)=F_i(1,x_1,\ldots,x_N)$ be the polynomial obtained by
setting $x_0$ to 1 in $F_i$. Then on $U_0$, $Z$ is defined by
the ideal $I'(Z)$ obtained from $I(Z)$ by setting $x_0=1$ in each
element of $I(Z)$, so $I'(Z)=(J')^m$ where $J'=(f_1,\ldots,f_N)$. 
Localizing at $P_1=I(p_1)=(x_1,\ldots,x_N)$ gives
$(P_1)_{P_1}^{m_1}=(I'(Z))_{P_1}=(J')^m_{P_1}$
and modding out by $P_1^{m_1+1}$ gives graded isomorphisms
$$\frac{P_1^{m_1}}{P_1^{m_1+1}}\cong
\frac{(P_1)_{P_1}^{m_1}}{(P_1)_{P_1}^{m_1+1}}=
\frac{(J')^m_{P_1}}{(P_1)_{P_1}^{m_1+1}}\cong
\frac{(J')^m+P_1^{m_1+1}}{(P_1)^{m_1+1}}.$$
Thus each $f_i$ can have no terms of degree less than
$d=m_1/m$ and there must be terms of degree exactly $d$, so $d$ is an integer,
hence $m_1=dm$ and $m\leq m_1$.
But the vector space dimension of $\frac{P_1^{m_1}}{P_1^{m_1+1}}$
is $\binom{m_1+N-1}{N-1}$, and the vector space dimension of
$\frac{(J')^m+P_1^{m_1+1}}{(P_1)^{m_1+1}}$ is at most
$\binom{m+N-1}{N-1}$, so we have $\binom{m_1+N-1}{N-1}\leq \binom{m+N-1}{N-1}$,
which implies $m_1\leq m$, hence $m=m_1$ and $d=1$.
With a change of coordinates, the 
same argument works for each point $p_i$, so $m_i=m$ for all $i$,
and, at each point $p_i$, the linear terms of the $f_j$ must
be linearly independent (otherwise we would have
$\dim_\KK\frac{(J')^m+P_1^{m_1+1}}{(P_1)^{m_1+1}}<\binom{m+N-1}{N-1}$).
Thus $J'=I(p_1+\cdots+p_r)$ on $U_0$, hence 
$J=I(p_1+\cdots+p_r)\subseteq\KK[\PP^N]$.
\end{remark}

\begin{question}
In Remark \ref{FatPtsAndCI2}, do we need to assume a priori that $J$ is homogeneous?
\end{question}

\section{Computational estimates of resurgences}

It might be possible to address some of the foregoing questions computationally.
More generally, it is of interest to consider to what extent quantities like resurgences can be computed.
 
\subsection{Denkert's thesis}
In an unpublished part of her thesis \cite{Den}, Denkert 
gives an algorithm for computing $\rho(I(Z))$ arbitrarily accurately
when the symbolic Rees algebra 
$$R_s(I(Z))=\oplus_t I(tZ)x^t\subseteq \KK[\PP^N][x]$$
is Noetherian (equivalently, when for some $a\geq 1$, all
powers of $I(aZ)$ are symbolic \cite[Theorem 2.1]{HHT}).

Let $I=I(Z)\subseteq \KK[\PP^N]$ be nontrivial and assume $I^{(at)} = (I^{(a)})^t$
for all $t > 0$. For each $s\geq1$, let $b_s$ be the largest $b$ such that
$I^{(asN)}\subseteq I^{b}$ and let $\epsilon>0$. Since $b_s\geq as$, 
we have $\frac{asN}{b_s}-\frac{asN}{b_s+1}=\frac{asN}{b_s(b_s+1)}\leq\frac{N}{as+1}$.
So, by picking $s\gg0$ and $\epsilon$ small, we can make
$\frac{N}{as+1}+\epsilon$ arbitrarily small. Denkert's algorithm either computes
$\rho(I)$ exactly or gives an estimate which is accurate to less than 
$\frac{N}{as+1}+\epsilon$.

Assume we have picked $s$ and $\epsilon$. Let $A=asN$ and let $B=b_s$. 
Thus we have 
$$I^{(At)}=(I^{(A)})^t\subseteq (I^{B})^t=I^{Bt}$$ 
for all $t\geq1$, and
for $m\geq At$ and $r\leq Bt$ we have 
$$I^{(m)}\subseteq I^{(At)}\subseteq I^{Bt}\subseteq I^r.$$

If $r\geq B\lceil\frac{A}{B\epsilon}\rceil$ and $\frac{m}{r}\geq\frac{A}{B}+\epsilon$, we now show for
the least $t$ such that $r<Bt$ that $m\geq At$, and hence $I^{(m)}\subseteq I^r$.
Indeed, we have $B(t-1)\leq r<Bt$, so $t-1=\lfloor\frac{r}{B}\rfloor\geq\lceil\frac{A}{B\epsilon}\rceil$.
Thus $(t-1)B\epsilon\geq A$, which is equivalent to
$(\frac{A}{B}+\epsilon)(t-1)B\geq At$, hence we have
$m\geq r(\frac{A}{B}+\epsilon)\geq B(t-1)(\frac{A}{B}+\epsilon)\geq At$.

In summary, we have $I^{(m)}\subseteq I^r$ for all 
$r\geq B\lceil\frac{A}{B\epsilon}\rceil$ and $\frac{m}{r}\geq\frac{A}{B}+\epsilon$,
and we have $I^{(A)}\not\subseteq I^{B+1}$.
Thus $\frac{A}{B+1}\leq \rho(I)$ and either $\rho(I)< \frac{A}{B}+\epsilon$
or $\rho(I)=m/r$ for some $r< B\lceil\frac{A}{B\epsilon}\rceil$
with $N>m/r\geq\frac{A}{B}+\epsilon$ (and there are only finitely many
such $r$ and $m$).

Moreover, since we have $I^{(m)}\subseteq I^r$ for all but finitely many $m$ and $r$
with $m/r\geq\frac{A}{B}+\epsilon$, we have $\widehat{\rho}(I) \leq \frac{A}{B}+\epsilon$
for each $\epsilon>0$, and hence we have $\widehat{\rho}(I) \leq \frac{A}{B}$.
This also follows by Theorem 1.2(3) of \cite{GHVT}.

\subsection{DiPasquale-Drabkin method}
An alternate approach is based on \cite{DD}.
Again assume the symbolic Rees algebra of a homogeneous ideal $I\subset\KK[\PP^N]$ is finitely 
generated. Then \cite{DD} shows that we can compute $\widehat{\rho}(I)$ exactly
(assuming we know the Rees valuations). Let $\epsilon>0$. Then, as in the proof of Theorem \ref{MT1},
whenever $I^{(m)}\not\subseteq I^r$, we have $\frac{s}{r+N} < \widehat{\rho}(I)$.
But there are only finitely many $r$ for which there is an $s$ such that both
$\frac{s}{r+N} < \widehat{\rho}(I)$ and
$\widehat{\rho}(I)+\epsilon \leq s/r$ hold.
For each of these $s$ and $r$ we check if $I^{(s)}$ is not contained in $I^r$.
For all such $r$ and $s$ which occur (if any), then $\rho(I)$ is the maximum of their ratios $s/r$. If none occur, then 
$\widehat{\rho}(I) \leq \rho(I) < \widehat{\rho}(I)+ \epsilon$. 

The estimates on $\rho(I)$ given by using \cite{DD} versus those of \cite{Den} are somewhat
similar, but each starts with different input data. Also, \cite{DD} gives 
us $\widehat{\rho}(I)$ exactly (if we know the Rees valuations), whereas 
\cite{Den} just gives us an upper bound on $\widehat{\rho}(I)$.



\begin{thebibliography}{CFGLMNSSV}

\bibitem[AH]{AH1}
I. M. Aberbach and C. Huneke, 
{\it An improved Brian{\c c}on–Skoda theorem with applications 
to the Cohen–Macaulayness of Rees algebras}, 
Math. Ann. 297 (1993), 343--369.


\bibitem[AM]{AM}
M.F. Atiyah and I.G. Macdonald,
Introduction to Commutative Algebra,
Addison-Wesley, 1969, ix+128 pp.

\bibitem[BZ]{BZ}
I.\ Bahmani Jafarloo and G.\ Zito,
{\it On the containment problem for fat points},
to appear, J. Commutative Algebra (arXiv:1802.10178v3).

\bibitem[BH]{BH1}
C.\ Bocci and B.\ Harbourne,
{\em Comparing powers and symbolic powers of ideals}, 
J. Algebraic Geometry, 19 (2010), 399--417 (arXiv:0706.3707).

\bibitem[BH2]{BH2}
C.\ Bocci and B.\ Harbourne,
{\it The resurgence of ideals of points and the containment problem},
Proc. Amer. Math. Soc., 138 (2010) 1175--1190 (arXiv:0906.4478).

\bibitem[Ch]{Ch}  G.\ V.\ Chudnovsky. {\it Singular points on complex 
hypersurfaces and multidimensional
Schwarz Lemma}, S\'eminaire de Th\'eorie des Nombres, Paris 1979--80,
S\'eminaire Delange-Pisot-Poitou, Progress in Math vol. 12, M-J Bertin, editor,
Birkh\"auser, Boston-Basel-Stutgart (1981).

\bibitem[CFGLMNSSV]{CFGLMNSSV}
S.\ Cooper, G.\ Fatabbi, E.\ Guardo, A.\ Lorenzini, J.\ Migliore, U.\ Nagel, 
A.\ Seceleanu, J.\ Szpond, A. Van Tuyl,
{\it Symbolic powers of codimension two Cohen-Macaulay ideals},
to appear, Comm. Alg. (arXiv:1606.00935v2)

\bibitem[CN]{CN}
R.C. Cowsik and M.V. Nori, 
{\it On the fibres of blowing up}, 
J. Indian Math. Soc. (N.S.) 40 (1976) 217--222.

\bibitem[Dem]{Dem}
J-P. Demailly, 
{\em Formules de Jensen en plusieurs variables et applications arithm\'etiques}, 
Bull. Soc. Math. France 110 (1982), 75-102.

\bibitem[Den]{Den}
A.\ Denkert,
{\em Results on containments and resurgences, with a focus on ideals of points in the plane},
PhD thesis, 2013, University of Nebraska-Lincoln.

\bibitem[DFMS]{DFMS}
M.\ DiPasquale, C.A.\ Francisco, J.\ Mermin, J.\ Schweig,
{\em Asymptotic resurgence via integral closures}, 
Trans. Amer. Math. Soc. 372 (2019), 6655--6676 (arXiv:1808.01547).

\bibitem[DD]{DD}
M.\ DiPasquale and B.\ Drabkin,
{\em The resurgence of an ideal with Noetherian symbolic Rees algebra is rational},
preprint, 2020 (arXiv:2003.06980).

\bibitem[DHNSST]{DHNSST}
M.\ Dumnicki, B.\ Harbourne, U.\ Nagel, A.\ Seceleanu, T.\ Szemberg and H.\ Tutaj-Gasi\'nska,
{\em Resurgences for ideals of special point configurations in $\PP^N$ coming from hyperplane arrangements},
J.\ Algebra 443 (2015) 383--394 (arXiv:1404.4957).

\bibitem[ELS]{ELS}
L.\ Ein, R.\ Lazarsfeld and K.\ Smith,
{Uniform bounds and symbolic powers on smooth varieties},
Invent. Math. 144 (2001), 241--252.

\bibitem[EV]{EV}
H.\ Esnault and E.\ Viehweg,
{\em Sur une minoration du degr\'e d'hypersurfaces s'annulant en certains points}, 
Math. Ann. 263 (1983), 75-86

\bibitem[GGSV]{GGSV}
F.\ Galetto, A.V.\ Geramita, Y.-S.\ Shin, A.\ Van Tuyl, 
{\it The symbolic defect of an ideal},
J. Pure Appl. Algebra 223,6 (2019) 2709--2731.

\bibitem[G]{G}
E.\ Grifo,
{\em A stable version of Harbourne's conjecture and the containment 
problem for space monomial curves},
preprint, 2018 (arXiv:1809.06955).

\bibitem[GHM]{GHM}
E.\ Grifo, C.\ Huneke and V.\ Mukundan,
{\it Expected resurgences and symbolic powers of ideals},
to appear, J. London Math. Soc. (arXiv:1903.12122v3)

\bibitem[GHVT]{GHVT}
E.\ Guardo, B.\ Harbourne and A.\ Van Tuyl,
{\em Asymptotic resurgences for ideals of positive dimensional 
subschemes of projective space},
Advances in Mathematics 246 (2013) 114--127 (arXiv:1202.4370).

\bibitem[HaHu]{HaHu}
B.\ Harbourne and C.\ Huneke,
{\em Are symbolic powers highly evolved?}
J. Ramanujan Math. Soc. 28:3 (2013) 311--330 (arXiv:1103.5809).

\bibitem[HHT]{HHT}
J.\ Herzog, T.\ Hibi, and N.V.\ Trung, 
{\em Symbolic powers of monomial ideals and vertex cover algebras},
Adv. Math. 210:1 (2017) 304-322.

\bibitem[HoHu]{HoHu}
M.\ Hochster and C.\ Huneke, 
{\em Comparison of symbolic and ordinary powers of ideals},
Invent. Math. 147 (2002), no. 2, 349--369.

\bibitem[HU]{HU}
C.\ Huneke and B.\ Ulrich, 
{\it Powers of licci ideals}, in ``Commutative Algebra, Berkeley, CA, 1987," 
MSRI Publ., Vol. 15, pp. 339--346, Springer-Verlag, New York, 1989.

\bibitem[SH]{HuSw}
I.\ Swanson and C.\ Huneke, 
Integral closure of ideals, rings, and modules, 
vol. 336 of London Mathematical Society Lecture Note Series. Cambridge University Press, Cambridge,
2006.

\bibitem[TX]{TX}
\c{S}.O.\ Toh\v{a}neanu and Y.\ Xie,
{\em On the containment problem for fat points ideals and Harbourne's conjecture},
to appear, Proc.\ Amer.\ Math.\ Soc.

\bibitem[ZS]{ZS}
O.\ Zariski and P.\ Samuel, Commutative Algebra, vol. II, Univ. Ser. Higher Math., D. Van Nostrand Co., Inc., Princeton, NJ, 
Toronto–London–New York, 1960.

\end{thebibliography}
\end{document}